\documentclass[12pt]{article}
\usepackage{amssymb}
\usepackage{amsthm}
\usepackage{amsmath}
\usepackage{graphicx}
\usepackage{enumerate}
\usepackage{pict2e}
\usepackage{framed}
\DeclareMathOperator{\Op}{Op}

\DeclareMathOperator{\Con}{Con}

\DeclareMathOperator{\Max}{Max}

\newtheorem{theorem}{Theorem}[section]

\newtheorem{lemma}[theorem]{Lemma}
\newtheorem{proposition}[theorem]{Proposition}

\newtheorem{remark}[theorem]{Remark}
\newtheorem{example}[theorem]{Example}
\newtheorem{corollary}[theorem]{Corollary}
\title{A generalization of pseudocomplementation in semilattices, lattices and posets with $0$}
\author{Ivan~Chajda and Helmut~L\"anger$^2$}
\date{}

\begin{document}
	
\maketitle

\footnotetext[1]{Support of the research of the first author by the Czech Science Foundation (GA\v CR), project 25-20013L, and support of the research of the second author by the Austrian Science Fund (FWF), project 10.55776/PIN5424624, is gratefully acknowledged.}
\footnotetext[2]{Corresponding author}
	
\begin{abstract}
Pseudocomplemented lattices satisfy a lot of nice properties and were studied by a number of algebraists. However, even when a poset $\mathbf P=(P,\le,0)$ with $0$ is not pseudocomplemented, we can assign to every element $a$ of $P$ the set $a^p$ of all maximal elements being orthogonal to $a$. This assignment $^p$ is a non-trivial operator provided $\mathbf P$ satisfies the so-called Ascending Chain Condition (ACC). In particular, we investigate this operator when the poset $\mathbf P$ in question is a semilattice or a lattice. We characterize this operator and describe its properties that turn out to be similar to the corresponding properties of pseudocomplementation. Moreover, we introduce an equivalence relation $\Theta_p$ on $P$ induced by $^p$ similarly as it was done by V.~I.~Glivenko for distributive pseudocomplemented lattices. Finally, we investigate for which lattices $\mathbf L$ with $0$ satisfying the ACC this equivalence relation turns out to be a congruence and when the quotient lattice $\mathbf L/\Theta_p$ is complemented respectively modular.
\end{abstract}



\section{Introduction}

Let $\mathbf P=(P,\le,0)$ be a poset with $0$ and $a,b\in L$. Throughout this paper all posets, semilattices and lattices are assumed to be non-empty. We write $a\perp b$ if $a\wedge b$ exists and is equal to $0$ and put $a^\perp:=\{x\in P\mid x\perp a\}$. The element $b$ is called a {\em pseudocomplement} of $a$ (see e.g.\ \cite B) if it is the greatest element of $a^\perp$. Such an element $b$ is usually denoted by $a^*$. The {\em poset} $\mathbf P$ is called {\em pseudocomplemented} if every of its elements has a pseudocomplement. From \cite F it is well-known that every complete completely distributive lattice, and hence every finite distributive lattice, is pseudocomplemented. However, there exist lattices that are pseudocomplemented, but not even modular (see \cite{CR}), e.g.\ the lattice $\mathbf N_5$ in Fig.~1:
\begin{center}
\setlength{\unitlength}{7mm}
\begin{picture}(6,8)
\put(3,1){\circle*{.3}}
\put(1,4){\circle*{.3}}
\put(5,3){\circle*{.3}}
\put(5,5){\circle*{.3}}
\put(3,7){\circle*{.3}}
\put(3,1){\line(-2,3)2}
\put(3,1){\line(1,1)2}
\put(3,7){\line(-2,-3)2}
\put(3,7){\line(1,-1)2}
\put(5,3){\line(0,1)2}
\put(2.85,.3){$0$}
\put(.35,3.85){$b$}
\put(5.4,2.85){$a$}
\put(5.4,4.85){$c$}
\put(2.85,7.4){$1$}
\put(.4,-.45){{\rm Figure~1. Lattice $\mathbf N_5$}}
\end{picture}
\end{center}

\vspace*{3mm}

Here we have
\[
\begin{array}{l|c|c|c|c|c}
x   & 0 & a & b & c & 1 \\
\hline
x^* & 1 & b & c & b & 0
\end{array}
\]
For modular pseudocomplemented semilattices the reader is referred to \cite C. On the other hand, there are modular lattices with $0$ that are not pseudocomplemented, see Fig.~2:

\vspace*{-3mm}

\begin{center}
\setlength{\unitlength}{7mm}
\begin{picture}(6,8)
\put(3,1){\circle*{.3}}
\put(1,3){\circle*{.3}}
\put(3,3){\circle*{.3}}
\put(5,3){\circle*{.3}}
\put(3,5){\circle*{.3}}
\put(3,7){\circle*{.3}}
\put(3,1){\line(-1,1)2}
\put(3,1){\line(0,1)6}
\put(3,1){\line(1,1)2}
\put(3,5){\line(-1,-1)2}
\put(3,5){\line(1,-1)2}
\put(2.85,.3){$0$}
\put(.35,2.85){$a$}
\put(3.4,2.85){$b$}
\put(5.4,2.85){$c$}
\put(3.4,4.85){$d$}
\put(2.85,7.4){$1$}
\put(-3.7,-.75){{\rm Figure~2. Not pseudocomplemented modular lattice}}
\end{picture}
\end{center}

\vspace*{4mm}

In this lattice $a^*$ does not exist. This motivated us to generalize the notion of a pseudocomplement to arbitrary lattices, semilattices and posets with $0$ satisfying the Ascending Chain Condition (ACC, for short) and to arbitrary subsets of them. The mentioned condition means that there do not exist infinite ascending chains.

Let $\mathbf P=(P,\le,0)$ be a poset with $0$ and $A,B\subseteq L$. Then we define
\begin{align*}
& A\perp B\text{ if }x\perp y\text{ for all }x\in A\text{ and all }y\in B, \\
& A^\perp:=\{x\in P\mid x\perp A\}, \\
& \Max A:=\text{set of all maximal elements of A}, \\
& A^p:=\Max(A^\perp), \\
& A\le_1 B\text{ if for all }x\in A\text{ there exists some }y\in B\text{ with }x\le y, \\
& A=_1B\text{ if both }A\le_1B\text{ and }B\le_1A.
\end{align*}
By an {\em operator} on $P$ we mean a mapping from $P$ to $2^P$. Let $\Op(M)$ denote the set of all operators on $P$. In the following we identify singletons with their unique element. Of course, for an element $a$ of a pseudocomplemented poset with $0$ we have $a^*=a^p$. Thus the operator $^p$ is in fact a generalization of the operation of pseudocomplementation. The pair $(^\perp,^\perp)$ is the Galois connection between $(2^P,\subseteq)$ and $(2^P,\subseteq)$ induced by the binary relation $\perp$ on $P$. The mapping $^p$ is an operator on $P$, but $(A^p,\le)$ is an antichain of $\mathbf P$ for every subset $A$ of $P$. In order to guarantee that $A^p\ne\emptyset$ whenever $A\ne\emptyset$ we will consider only posets with $0$ satisfying the ACC.

For the lattice in Fig.~2 we have the following values of the operators $^p$ and $^{pp}$:
\[
\begin{array}{l|c|c|c|c|c|c}
x      & 0 &    a    &    b    &    c    & d & 1 \\
\hline
x^p    & 1 & \{b,c\} & \{a,c\} & \{a,b\} & 0 & 0 \\
\hline
x^{pp} & 0 &    a    &    b    &    c    & 1 & 1	
\end{array}
\]
The aim of this paper is to show similarities of the operator $^p$ with pseudocomplementation. The advantage of this approach is that this operator can be introduced in every lattice, semilattice or a poset with $0$ satisfying the ACC and hence it is really an extension of pseudocomplementation.

\section{Properties of the operator of generalized pseudocomplementation}

We start with two elementary technical lemmas.

\begin{lemma}\label{lem1}
Let $(P,\le,0)$ be a poset with $0$ satisfying the {\rm ACC} and $A\subseteq P$. Then the following hold:
\begin{enumerate}[{\rm(i)}]
\item $(\Max A)^\perp=A^\perp$,
\item $A^{pp}=\Max(A^{\perp\perp})$ and $A^{ppp}=\Max(A^{\perp\perp\perp})$.
\end{enumerate}
\end{lemma}

\begin{proof}
\begin{enumerate}[(i)]
\item Since $\Max A\subseteq A$ we have $A^\perp\subseteq(\Max A)^\perp$. Now let $a\in(\Max A)^\perp$ and suppose $b\in A$. Since $^p$ satisfies the ACC there exists some $c\in\Max A$ with $b\le c$. Because of $a\in(\Max A)^\perp$ we have $a\wedge c=0$ and therefore $a\wedge b=0$. Since $b$ was an arbitrary element of $A$ we conclude $a\in A^\perp$ and hence $(\Max A)^\perp\subseteq A^\perp$.	
\item Using (i) we obtain
\begin{align*}
 A^{pp} & =(A^p)^p=\Max\Big(\big(\Max(A^\perp)\big)^\perp\Big)=\Max\big((A^\perp)^\perp\big)=\Max(A^{\perp\perp}), \\
A^{ppp} & =(A^{pp})^p=\Max\Big(\big(\Max(A^{\perp\perp})\big)^\perp\Big)=\Max\big((A^{\perp\perp})^\perp\big)=\Max(A^{\perp\perp\perp}).
\end{align*}
\end{enumerate}
\end{proof}

Since we work in the sequel with antichains, the following simplification can be useful.

\begin{lemma}\label{lem2}
Let $\mathbf P=(P,\le)$ be a poset and $(A,\le),(B,\le)$ antichains of $\mathbf P$ and assume $A=_1B$. Then $A=B$.
\end{lemma}

\begin{proof}
Assume $a\in A$. Because of $A\le_1B$ there exists some $b\in B$ with $a\le b$. Since $B\le_1A$ there exists some $c\in A$ with $b\le c$. Now $a,c\in A$ and $a\le b\le c$. Because $(A,\le)$ is an antichain of $\mathbf P$ we conclude $a=b\in B$. This shows $A\subseteq B$. The inclusion $B\subseteq A$ follows by symmetry.
\end{proof}

If $(P,\le,0)$ is a poset with $0$ and $A,B\subseteq S$ then we define
\[
A\wedge B=0\text{ if }x\wedge y=0\text{ for all }x\in A\text{ and all }y\in B.
\]
If $(S,\wedge)$ is a meet-semilattice and $A,B\subseteq S$ then we define
\[
A\wedge B:=\{x\wedge y\mid x\in A,y\in B\}.
\]

The basic properties of the operator $^p$ on a poset with $0$ are listed in the following proposition.

\begin{proposition}\label{prop2}
Let $\mathbf P=(P,\le,0)$ be a poset with $0$ satisfying the {\rm ACC} and $A,B\subseteq P$. Then
\begin{enumerate}[{\rm(i)}]
\item $(A^p,\le)$ is an antichain, $A\wedge A^p=0$ provided $A\ne\emptyset$, $\emptyset^p=0^p=\Max P$ and $\emptyset^{pp}=0^{pp}=0$,
\item $A\le_1A^{pp}$ and $A^{ppp}=A^p$,
\item $A\le_1B$ implies $B^p\le_1A^p$ and $A^{pp}\le_1B^{pp}$,
\item the following are equivalent: $A\le_1B^p$, $B\le_1A^p$ and $A\subseteq B^\perp$,
\item $A^{pp}=A$ if and only if there exists some $B\subseteq P$ with $A=B^p$,
\item $A^p\le_1B$ if and only if $A^\perp\le_1B$,
\item $A^p\le_1B^p$ if and only if $A^\perp\subseteq B^\perp$, and $A^p=B^p$ if and only if $A^\perp=B^\perp$.	
\end{enumerate}
\end{proposition}

\begin{proof}
\begin{enumerate}[(i)]
\item This follows from the definition of $A^p$.	
\item Since $(^\perp,^\perp)$ is a Galois connection between $(2^P,\subseteq)$ and $(2^P,\subseteq)$, we have $A\subseteq A^{\perp\perp}$. Due to the ACC and Lemma~\ref{lem1} (ii) we conclude $A\le_1\Max(A^{\perp\perp})=A^{pp}$. Again according to Lemma~\ref{lem1} (ii) we have $A^{ppp}=\Max(A^{\perp\perp\perp})=\Max(A^\perp)=A^p$.
\item Assume $A\le_1B$, $a\in B^p$ and $b\in A$. Since $A\le_1B$ there exists some $c\in B$ with $b\le c$. Because of $a\in B^p\subseteq B^\perp$ we have $a\wedge c=0$ and hence $a\wedge b=0$. Since $b$ was an arbitrary element of $A$, this shows $a\in A^\perp$. Because of the ACC there exists some $d\in A^p$ with $a\le d$. This shows $B^p\le_1A^p$. Now apply the implication just proved to $B^p\le_1A^p$.
\item If $A\le_1B^p$ then by (ii) and (iii) we get $B\le_1B^{pp}\le_1A^p$ and therefore $B\le_1A^p$. Now assume $b\in B\le_1A^p$. Then there exists some $a\in A^p$ with $b\le a$ and because of $A^p\subseteq A^\perp$ we have $a\in A^\perp$ and since $b\le a$ we conclude $b\in A^\perp$. This shows $B\subseteq A^\perp$ and therefore $A\subseteq B^\perp$. Finally, assume $A\subseteq B^\perp$. If $a\in A$ then $a\in B^\perp$ and because of the ACC there exists some $b\in B^p$ satisfying $a\le b$. This shows $A\le_1B^p$.
\item If $A^{pp}=A$ then $A=(A^p)^p$. If, conversely, there exists some $B\subseteq P$ with $A=B^p$ then $A^{pp}=B^{ppp}=B^p=A$ according to (ii).
\item Since $A^p\subseteq A^\perp$, $A^\perp\le_1B$ implies $A^p\le_1B$. Conversely, assume $A^p\le_1B$. Further, assume $a\in A^\perp$. Then according to the ACC there exists some $b\in A^p$ with $a\le b$. Because of $A^p\le_1B$ there exists some $c\in B$ with $b\le c$. This shows $a\le c$ and hence $A^\perp\le_1B$.
\item Because of (vi) and (v) the following are equivalent:
\[
A^p\le_1B^p,\enspace A^\perp\le_1B^p,\enspace A^\perp\subseteq B^\perp.
\]
\item From this we obtain that $A^p=_1B^p$ is equivalent to $A^\perp=B^\perp$. Now (viii) follows from Lemma~\ref{lem2}.
\end{enumerate}		
\end{proof}

It is worth noticing that inequality in (ii) of Proposition~\ref{prop2} cannot be replaced by the corresponding equality. Namely, in the lattice from Fig.~2 we have $d^{pp}=1\ne d$.

\begin{corollary}\label{cor1}
Let $\mathbf P=(P,\le,0)$ be a poset with $0$ satisfying the {\rm ACC} and $a,b\in P$. Then
\begin{enumerate}[{\rm(i)}]
\item $(a^p,\le)$ is an antichain and $a\wedge a^p=0$,
\item $a\le_1a^{pp}$ and $a^{ppp}=a^p$,
\item $a\le_1b$ implies $b^p\le_1a^p$ and $a^{pp}\le_1b^{pp}$,
\item the following are equivalent: $a\le_1b^p$, $b\le_1a^p$ and $a\subseteq b^\perp$,
\item $a^p\le_1b$ if and only if $a^\perp\le_1b$,
\item $a^p\le_1b^p$ if and only if $a^\perp\subseteq b^\perp$, and $a^p=b^p$ if and only if $a^\perp=b^\perp$.	
\end{enumerate}
\end{corollary}

\begin{remark}
Let $\mathbf P=(P,\le,0)$ be a poset with $0$ satisfying the {\rm ACC} and $a,b\in P$ and assume $a\le b$ and $a^p=0$. Then according to {\rm(iii)} of Corollary~\ref{cor1} we conclude $b^p\le_1a^p=0$ and hence $b^p=0$. This shows that $\{x\in P\mid x^p=0\}$ is an order filter of $\mathbf P$.
\end{remark}

If the poset in question is a meet-semilattice, we can prove the following.

\begin{lemma}\label{lem4}
Let $(S,\wedge,0)$ be a meet-semilattice with $0$ satisfying the {\rm ACC} and $A,B\subseteq S$ then the following hold:
\begin{enumerate}[{\rm(i)}]
\item If $A,B\ne\emptyset$ then $A\wedge(A\wedge B)^p=_1A\wedge B^p$,
\item $(A\wedge B)^{pp}\le_1A^{pp}\wedge B^{pp}$.
\end{enumerate}
\end{lemma}

\begin{proof}
\begin{enumerate}[(i)]
\item Assume $A,B\ne\emptyset$. Since $A\wedge B\le_1B$ we have $B^p\le_1(A\wedge B)^p$ according to (iii) of Proposition~\ref{prop2} which implies $A\wedge B^p\le_1A\wedge(A\wedge B)^p$. On the other hand, $A\wedge B\wedge(A\wedge B)^p=0$ according to (vi) of Proposition~\ref{prop2} whence $A\wedge(A\wedge B)^p\subseteq B^\perp$ and hence $A\wedge(A\wedge B)^p\le_1B^p$. Therefore $A\wedge(A\wedge B)^p\le_1A\wedge B^p$. Together we obtain $A\wedge(A\wedge B)^p=_1A\wedge B^p$.
\item Using Proposition~\ref{prop2} one sees that any of the following statements implies the next one:
\begin{align*}
& A\wedge B\le_1A\text{ and }A\wedge B\le_1B, \\
& (A\wedge B)^{pp}\le_1A^{pp}\text{ and }(A\wedge B)^{pp}\le_1B^{pp}, \\
& (A\wedge B)^{pp}\le_1A^{pp}\wedge B^{pp}.
\end{align*}
\end{enumerate}
\end{proof}

That in (ii) of Lemma~\ref{lem4} ``$\le_1$'' cannot be replaced by ``$=$'' can be seen as follows: In the lattice from Fig.~2 we have
\begin{align*}
(\{a,b\}\wedge\{a,c\})^{pp} & =\{0,a\}^{pp}=\{b,c\}^p=a\ne\{0,a\}=\{a,b\}\wedge\{a,c\}=c^p\wedge b^p= \\
                            & =\{a,b\}^{pp}\wedge\{a,c\}^{pp}.
\end{align*}
If $(S,\vee)$ is a join-semilattice and $A,B\subseteq S$ then we define
\[
A\vee B:=\{x\vee y\mid x\in A,y\in B\}.
\]
If the poset is even a lattice, we can prove a result similar to that of Lemma~\ref{lem4}, but for the operation $\vee$ instead of the operation $\wedge$.

\begin{lemma}\label{lem5}
Let $(L,\vee,\wedge,0)$ be a lattice with $0$ satisfying the {\rm ACC} and $\emptyset\ne A,B\subseteq L$. Then $(A\vee B)^p\le_1A^p\wedge B^p$.
\end{lemma}

\begin{proof}
Using Proposition~\ref{prop2} one sees that any of the following statements implies the next one:
\begin{align*}
& A\le_1A\vee B\text{ and }B\le_1A\vee B, \\
& (A\vee B)^p\le_1A^p\text{ and }(A\vee B)^p\le_1B^p, \\
& (A\vee B)^p\le_1A^p\wedge B^p.
\end{align*}	
\end{proof}

It is worth noticing that the inequality in Lemma~\ref{lem5} cannot be replaced by the corresponding equality. Namely, in the lattice from Fig.~2 we have:
\[
(a\vee b)^p=d^p=0\ne\{0,c\}=\{b,c\}\wedge\{a,c\}=a^p\wedge b^p.
\]

\section{A characterization of the operator $^p$ in semilattices and posets}

Pseudocomplemented semilattices were investigated in \cite J, see also \cite{CHK}.

We can characterize the operator $^p$ on bounded meet-semilattices satisfying the ACC as follows:

\begin{theorem}
Let $\mathbf S=(S,\wedge,0,1)$ be a bounded meet-semilattice satisfying the {\rm ACC} and $^p\colon S\to2^S$ an operator on $S$. Then $^p$ is the operator-pseudocomplementation on $\mathbf S$ if and only if the following conditions hold for all $x,y\in S$:
\begin{enumerate}[{\rm(i)}]
\item $(x^p,\le)$ is an antichain,
\item $x\wedge0^p=x$,
\item $x\wedge x^p=0$,
\item $x\wedge(x\wedge y)^p=_1x\wedge y^p$.
\end{enumerate}
\end{theorem}

\begin{proof}
Let $a,b\in S$. (i), (iii) and (iv) hold according to Corollary~\ref{cor1}, and (ii) holds since $0^p=1$. Conversely, assume (i) -- (iv) to hold. From (ii) we obtain $0^p=1$. Assume $a\in b^\perp$. Then $a\wedge b=0$ and hence $a=_1a\wedge b^p$ according to (iv) and (ii). This shows $a\le_1b^p$. Conversely, assume $a\le_1b^p$. Then $a=_1a\wedge b^p$ and hence
\[
a\wedge b=_1(a\wedge b^p)\wedge b=a\wedge(b^p\wedge b)=a\wedge0=0
\]
according to (iii). But this implies $a\wedge b=0$, i.e.\ $a\in b^\perp$. Hence $a\in b^\perp$ is equivalent to $a\le_1b^p$. If $a\in b^p$ then $a\le_1b^p$ and hence $a\in b^\perp$. This shows $b^p\subseteq b^\perp$. Now assume $a\in b^p$, $c\in b^\perp$ and $a\le c$. Since $c\in b^\perp$ we have $c\le_1b^p$, i.e.\ there exists some $d\in b^p$ with $c\le d$. Because of $a\le c\le d$ we have $a\le d$. Now (i) implies $a=d$. Therefore $a=c$. This shows $b^p\subseteq\Max(b^\perp)$. Conversely, assume $a\in\Max(b^\perp)$. Then $a\in b^\perp$ and hence $a\le_1b^p$, i.e.\ there exists some $c\in b^p$ with $a\le c$. Now $a\le c\in b^\perp$. Since $a\in\Max(b^\perp)$ we conclude $a=c\in b^p$ showing $\Max(b^\perp)\subseteq b^p$. Therefore $b^p=\Max(b^\perp)$ and $^p$ is the operator pseudocomplementation on $\mathbf S$.
\end{proof}

In the next proposition we consider the question when $^p$ is an involution.

\begin{proposition}\label{prop1}
Let $(P,\le,0)$ be a poset with $0$ satisfying the {\rm ACC} and $a\in P$. Then the following hold:
\begin{enumerate}[{\rm(i)}]
\item If there exists some $b\in P$ satisfying $a=b^p$ then $a^{pp}=a$,
\item if there exists some $b\in P$ satisfying $b\not\le a$ and $a^p\le_1b^p$ then $a^{pp}\ne a$,
\item if there exists some $b\in P$ satisfying $a<b$ and $a^p=b^p$ then $a^{pp}\ne a$.
\end{enumerate}
\end{proposition}

\begin{proof}
\begin{enumerate}[(i)]
\item If there exists some $b\in P$ satisfying $a=b^p$ then $a^{pp}=b^{ppp}=b^p=a$ according to Corollary~\ref{cor1}.
\item If there exists some $b\in P$ satisfying $b\not\le a$ and $a^p\le_1b^p$ then $a^{pp}=a$ would imply $b\le_1b^{pp}\le_1a^{pp}=a$ according to Corollary~\ref{cor1} and hence $b\le_1a$, i.e.\ $b\le a$ contradicting $b\not\le a$.
\item If there exists some $b\in P$ satisfying $a<b$ and $a^p=b^p$ then $a^{pp}\ne a$ according to (ii).
\end{enumerate}
\end{proof}

In the lattice from Fig.~2 we have $d<1$, $d^p=0=1^p$ and $d^{pp}=1\ne d$ in accordance with (iii) of Proposition~\ref{prop1}.

The following auxiliary result is expected but not trivial.

\begin{lemma}
Let $\mathbf S=(S,\vee)$ be a join-semilattice satisfying the {\rm ACC}. Then $\mathbf S$ has a greatest element.	
\end{lemma}

\begin{proof}
Since $S\ne\emptyset$ and $\mathbf S$ satisfies the ACC we have $\Max S\ne\emptyset$. If $a,b\in\Max S$ then because of $a,b\le a\vee b$ we conclude $a=a\vee b=b$. Thus $|\Max S|=1$, say $\Max S=\{c\}$. Let $d\in S$. Then because of the ACC there exists some $e\in\Max S$ with $d\le e$. Since $\Max S=\{c\}$ we conclude $d\le e=c$. This shows that $c$ is the greatest element of $\mathbf S$.	
\end{proof}

The operator $^p$ on a poset with $0$ satisfying the ACC can be characterized as follows.

\begin{theorem}\label{th2}
Let $\mathbf P=(P,\le,0)$ be a poset with $0$ satisfying the {\rm ACC} and $^p\colon P\to2^P$ an operator on $P$. Then $^p$ is the operator pseudocomplementation on $\mathbf P$ if and only if the following conditions 
hold for all $x,y\in P$:	
\begin{enumerate}[{\rm(i)}]
\item $(x^p,\le)$ is an antichain,
\item $x\wedge y=0$ if and only if there exists some $z\in y^p$ with $x\le z$.
\end{enumerate}
\end{theorem}

Observe that condition {\rm(ii)} can be written in the following way:
\begin{enumerate}
\item[(ii')] $x\in y^\perp$ if and only if $x\le_1x^p$.
\end{enumerate}	

\begin{proof}[Proof of Theorem~\ref{th2}]
If $^p$ is the operator pseudocomplementation on $\mathbf P$ then condition (i) holds because of Corollary~\ref{cor1} and condition (ii) holds because of (x) of Proposition~\ref{prop2}. Conversely, assume (i) and (ii) to hold. Let $a\in P$. Further assume $b\in a^p$. Then $b\le_1a^p$ and hence $b\in a^\perp$ by (ii). This shows $a^p\subseteq a^\perp$. Now assume $c\in a^\perp$ and $c\ge b$. Because of (ii) we have $c\le_1a^p$ and therefore there exists some $d\in a^p$ with $c\le d$. Now we have $b\le c\le d$. Since $b,d\in a^p$ we conclude by (i) that $b=d$ and hence $c=b$. This shows $b\in\Max(a^\perp)$. We have proved $a^p\subseteq\Max(a^\perp)$. Now assume $e\in\Max(a^\perp)$. Then $e\in a^\perp$ and hence by (ii) we conclude $e\le_1a^p$. This shows that there exists some $f\in a^p$ with $e\le f$. Now $f\le_1a^p$ which according to (ii) implies $f\in a^\perp$. Together with $e\in\Max(a^\perp)$ and $e\le f$ this yields $e=f\in a^p$ showing $\Max(a^\perp)\subseteq a^p$. Altogether we have proved $a^p=\Max(a^\perp)$ completing the proof of the theorem.
\end{proof}

A {\em lattice} $(L,\vee,\wedge,0)$ {\em with $0$} is called {\em $0$-distributive} if
\[
x,y,z\in L\text{ and }x\wedge y=x\wedge z=0\text{ together imply }x\wedge(y\vee z)=0.
\]
A {\em complete lattice} $(L,\vee,\wedge)$ is called {\em completely $0$-distributive} if
\[
x\in L,\emptyset\ne A\subseteq L\text{ and }x\wedge y=0\text{ for all }y\in A\text{ together imply }x\wedge\bigvee A=0.
\]

\begin{remark}
It is well-known that every completely $0$-distributive complete lattice is pseudocomplemented. This is also true for $0$-distributive lattices with $0$ satisfying the {\rm ACC} which can be seen as follows: Let $\mathbf L=(L,\vee,\wedge,0)$ be a $0$-distributive lattice with $0$ satisfying the {\rm ACC} and $a\in L$. According to the {\rm ACC}, $a^p\ne\emptyset$. Now let $b,c\in a^p$. Since $\mathbf L$ is $0$-distributive, we have $b\vee c\in a^\perp$. Together with $b,c\le b\vee c$ and $b,c\in a^p$ we conclude $b=b\vee c=c$ showing $|a^p|=1$, say $a^p=\{d\}$. Since $\mathbf L$ satisfies the {\rm ACC}, $d$ is the greatest element of $(a^\perp,\le)$. Because $a$ was an arbitrary element of $L$, $\mathbf L$ is pseudocomplemented. \\
On the other hand, there exist $0$-distributive complete lattices that are not pseudocomplemented and hence not completely $0$-distributive. This is shown by the following example. Let $\mathbb N$ denote the set of all non-negative integers and put
\[
L:=\{A\mid A\text{ a finite subset of }\mathbb N\}\cup\{\mathbb N\}\text{ and }\mathbf L:=(L,\cup,\cap).
\]
Then $\mathbf L$ is a complete lattice which is obviously $0$-distributive. But only the smallest and the greatest element of $L$ have a pseudocomplement and hence $\mathbf L$ is not pseudocomplemented. This is due to the fact that this lattice does not satisfy the {\rm ACC}.
\end{remark}

\section{Glivenko equivalence}

It was shown already by V.~I.~Glivenko \cite G that in a pseudocomplemented distributive lattice $\mathbf L=(L,\vee,\wedge,0)$ with $0$ the binary relation $\Theta$ given by
\[
(x,y)\in\Theta\text{ if }x^*=y^*
\]
($x,y\in L$) is a congruence on $\mathbf L$ and the quotient lattice $\mathbf L/\Theta$ is a Boolean algebra. We are going to modify this result for posets, semilattices and lattices with $0$ satisfying the ACC.

On a poset $(P,\le,0)$ with $0$ satisfying the ACC we define a binary relation $\Theta_p$ as follows:
\[
x\mathrel{\Theta_p}y\text{ if }x^p=y^p.
\]
($x,y\in P$). This is further equivalent to $x^\perp=y^\perp$ according to (x) of Corollary~\ref{cor1}. The relation $\Theta_p$ is an equivalence relation on $P$, the so-called {\em Glivenko equivalence} which is a certain generalization of the classical Glivenko equivalence, see e.g.\ \cite G.

\begin{lemma}\label{lem3}
Let $\mathbf S=(S,\wedge,0)$ be a meet-semilattice with $0$ satisfying the {\rm ACC}. Then $\Theta_p$ is compatible with $\wedge$ and hence it is a congruence on $\mathbf S$.
\end{lemma}

\begin{proof}
Assume $(a,b)\in\Theta_p$ and $c,d\in S$. Then $a^\perp=b^\perp$ and the following are equivalent:
\[
d\in(a\wedge c)^\perp,\enspace a\wedge c\wedge d=0,\enspace c\wedge d\in a^\perp,\enspace c\wedge d\in b^\perp,\enspace b\wedge c\wedge d=0,\enspace d\in(b\wedge c)^\perp.
\]
This shows $(a\wedge c)^\perp=(b\wedge c)^\perp$, i.e. $(a\wedge c,b\wedge c)\in\Theta_p$.
\end{proof}

\begin{remark}
There are lattices $\mathbf L=(L,\vee,\wedge,0)$ with $0$ satisfying the {\rm ACC} where $\Theta_p$ is not a congruence on $\mathbf L$, see Fig.~3:

\vspace*{-3mm}

\begin{center}
\setlength{\unitlength}{7mm}
\begin{picture}(6,8)
\put(3,1){\circle*{.3}}
\put(1,3){\circle*{.3}}
\put(3,3){\circle*{.3}}
\put(5,3){\circle*{.3}}
\put(1,5){\circle*{.3}}
\put(3,5){\circle*{.3}}
\put(3,7){\circle*{.3}}
\put(3,1){\line(-1,1)2}
\put(3,1){\line(0,1)6}
\put(3,1){\line(1,1)2}
\put(3,7){\line(-1,-1)2}
\put(3,7){\line(1,-2)2}
\put(1,5){\line(0,-1)2}
\put(3,5){\line(-1,-1)2}
\put(2.85,.3){$0$}
\put(.35,2.85){$a$}
\put(3.4,2.85){$b$}
\put(5.4,2.85){$c$}
\put(.35,4.85){$d$}
\put(3.4,4.85){$e$}
\put(2.85,7.4){$1$}
\put(-6,-.75){{\rm Figure~3. Lattice where the Glivenko equivalence is not a congruence}}
\end{picture}
\end{center}

\vspace*{4mm}

The operator $^p$ is given by the following table:
\[
\begin{array}{l|c|c|c|c|c|c|c}
x   & 0 &    a    & b &    c    &    d    & e & 1 \\
\hline
x^p & 1 & \{b,c\} & c & \{d,e\} & \{b,c\} & c & 0
\end{array}
\]	
We have $(a,d)\in\Theta_p$ since $a^p=\{b,c\}=d^p$, but $(a\vee b,d\vee b)\notin\Theta_p$ since $(a\vee b)^p=e^p=c\ne0=1^p=(d\vee b)^p$.
\end{remark}

In the next theorem we characterize those lattices with $0$ satisfying the ACC for which the Glivenko equivalence is a congruence.

\begin{theorem}
Let $\mathbf L=(L,\vee,\wedge,0)$ be a lattice with $0$ satisfying the {\rm ACC}. Then $\Theta_p$ is a congruence on $\mathbf L$ if and only if for arbitrary $a,b,c,d\in L$ the following two conditions together imply $(b\vee c)\wedge d=0$:
\begin{enumerate}[{\rm(i)}]
\item for all $x\in L$, $a\wedge x=0$ is equivalent to $b\wedge x=0$,
\item $(a\vee c)\wedge d=0$.
\end{enumerate}
\end{theorem}

\begin{proof}
It is evident that if $\Theta_p$ is a congruence on $\mathbf L$ then the condition is satisfied. We prove that also the converse holds. The implication says: For arbitrary $a,b,c,d\in L$, $a^\perp=b^\perp$ and $d\in(a\vee c)^\perp$ together imply $d\in(b\vee c)^\perp$. This means that $(a,b)\in\Theta_p$ implies $(a\vee c)^\perp\subseteq(b\vee c)^\perp$. Because of the symmetry of $\Theta_p$ this says that $(a,b)\in\Theta_p$ implies $(a\vee c)^\perp=(b\vee c)^\perp$, i.e.\ $(a\vee c,b\vee c)\in\Theta_p$. Lemma~\ref{lem3} completes the proof of the theorem.
\end{proof}

The next example shows that also in the case when $\Theta_p$ is a congruence on $\mathbf L$, the quotient lattice $\mathbf L/\Theta_p$ need neither be complemented nor modular.

\begin{example}
Consider the lattice $\mathbf L=(L,\vee,\wedge)$ depicted in Fig.~4:

\vspace*{-3mm}

\begin{center}
\setlength{\unitlength}{7mm}
\begin{picture}(6,8)
\put(3,1){\circle*{.3}}
\put(1,3){\circle*{.3}}
\put(3,3){\circle*{.3}}
\put(5,3){\circle*{.3}}
\put(1,5){\circle*{.3}}
\put(5,5){\circle*{.3}}
\put(3,7){\circle*{.3}}
\put(3,1){\line(-1,1)2}
\put(3,1){\line(0,1)2}
\put(3,1){\line(1,1)2}
\put(3,7){\line(-1,-1)2}
\put(3,7){\line(1,-1)2}
\put(3,3){\line(-1,1)2}
\put(3,3){\line(1,1)2}
\put(1,3){\line(0,1)2}
\put(5,3){\line(0,1)2}
\put(2.85,.3){$0$}
\put(.35,2.85){$a$}
\put(3.4,2.85){$b$}
\put(5.4,2.85){$c$}
\put(.35,4.85){$d$}
\put(5.4,4.85){$e$}
\put(2.85,7.4){$1$}
\put(-4.2,-.75){{\rm Figure~4. Not pseudocomplemented non-modular lattice}}
\end{picture}
\end{center}

\vspace*{4mm}

$\mathbf L$ is not pseudocomplemented since $b$ has no pseudocomplement, $\mathbf L$ is not complemented since $b$ has no complement and $\mathbf L$ is not modular since the sublattice $(\{0,a,c,d,1\},\vee,\wedge)$ of $\mathbf L$ is isomorphic to $\mathbf N_5$. The operator $^p$ is given by the following table:
\[
\begin{array}{l|c|c|c|c|c|c|c}
x   & 0 & a &    b    & c & d & e & 1 \\
\hline
x^p & 1 & e & \{a,c\} & d & c & a & 0
\end{array}
\]	
This shows that $\Theta_p=\Delta\in\Con\mathbf L$ and $\mathbf L/\Theta_p\cong\mathbf L$. Hence $\mathbf L/\Theta_p$ is neither complemented nor modular.
\end{example}

Finally, we characterize those bounded lattices $\mathbf L$ satisfying the ACC for which the quotient lattice $\mathbf L/\Theta_p$ is complemented.

\begin{theorem}\label{th1}
Let $\mathbf L=(L,\vee,\wedge,0,1)$ be a bounded lattice satisfying the {\rm ACC} such that $\Theta_p\in\Con\mathbf L$. Then the following are equivalent:
\begin{enumerate}[{\rm(i)}]
\item $\mathbf L/\Theta_p$ is complemented,
\item to every $x\in L$ there exists some $y\in x^p$ satisfying $(x\vee y)^p=0$.
\end{enumerate}
\end{theorem}

\begin{proof}
Let $a,b\in L$. Then $[a]\Theta_p\vee[b]\Theta_p=[1]\Theta_p$ if and only if $(a\vee b)^p=0$. Moreover, the following are equivalent:
\[
[a]\Theta_p\wedge[b]\Theta_p=[0]\Theta_p,\enspace(a\wedge b)^p=1,\enspace(a\wedge b)^{pp}=0,\enspace a\wedge b=0,\enspace b\in a^\perp.
\]
Observe that $a\le b$ together with $a^p=0$ implies $b^p\le_1a^p=0$, i.e.\ $b^p=0$. Now the following statements are equivalent:
\begin{align*}
& {\rm(i)}, \\
& \text{to every }x\in L\text{ there exists some }z\in L\text{ satisfying } [x]\Theta_p\vee[z]\Theta_p=[1]\Theta_p\text{ and} \\
& \hspace*{1cm}[x]\Theta_p\wedge[z]\Theta_p=[0]\Theta_p, \\
& \text{to every }x\in L\text{ there exists some }z\in L\text{ satisfying }(x\vee z)^p=0\text{ and }z\in x^\perp, \\
& \text{to every }x\in L\text{ there exists some }z\in x^\perp\text{ satisfying }(x\vee z)^p=0, \\
& {\rm(ii)}.
\end{align*}
\end{proof}

\begin{example}
\begin{enumerate}[{\rm(i)}]
\item Since for the lattice from Figure~4 we have
\begin{align*}
        b^p & =\{a,c\}, \\
(b\vee a)^p & =d^p=c\ne0, \\
(b\vee c)^p & =e^p=a\ne0,
\end{align*}
$\mathbf L/\Theta_p\cong\mathbf L$ is not complemented according to Theorem~\ref{th1}.
\item Consider the not complemented lattice $\mathbf L$ visualized in Fig.~5:

\vspace*{-3mm}

\begin{center}
\setlength{\unitlength}{7mm}
\begin{picture}(6,10)
\put(3,1){\circle*{.3}}
\put(1,3){\circle*{.3}}
\put(3,3){\circle*{.3}}
\put(5,3){\circle*{.3}}
\put(3,5){\circle*{.3}}
\put(3,7){\circle*{.3}}
\put(3,9){\circle*{.3}}
\put(3,1){\line(-1,1)2}
\put(3,1){\line(0,1)8}
\put(3,1){\line(1,1)2}
\put(1,3){\line(1,1)2}
\put(5,3){\line(-1,2)2}
\put(2.85,.3){$0$}
\put(.35,2.85){$a$}
\put(3.4,2.85){$b$}
\put(5.4,2.85){$c$}
\put(3.4,4.85){$d$}
\put(3.4,6.85){$e$}
\put(2.85,9.4){$1$}
\put(-5,-.75){{\rm Figure~5. Not pseudocomplementerd not complemented lattice}}
\end{picture}
\end{center}

\vspace*{4mm}

The operator $^p$ is given by the following table:
\[
\begin{array}{l|c|c|c|c|c|c|c}
x   & 0 &    a    &    b    & c & d & e & 1 \\
\hline
x^p & 1 & \{b,c\} & \{a,c\} & d & c & 0 & 0
\end{array}
\]
Hence $\Theta_p=\Delta\cup\{e,1\}^2\in\Con\mathbf L$. However, the lattice $\mathbf L/\Theta_p$ depicted in Fig.~6 is complemented:

\vspace*{-3mm}

\begin{center}
\setlength{\unitlength}{7mm}
\begin{picture}(6,8)
\put(3,1){\circle*{.3}}
\put(1,3){\circle*{.3}}
\put(3,3){\circle*{.3}}
\put(5,3){\circle*{.3}}
\put(3,5){\circle*{.3}}
\put(3,7){\circle*{.3}}
\put(3,1){\line(-1,1)2}
\put(3,1){\line(0,1)6}
\put(3,1){\line(1,1)2}
\put(1,3){\line(1,1)2}
\put(5,3){\line(-1,2)2}
\put(2.7,.3){$[0]$}
\put(.1,2.85){$[a]$}
\put(3.25,2.85){$[b]$}
\put(5.25,2.85){$[c]$}
\put(3.25,4.85){$[d]$}
\put(2.65,7.4){$[1]$}
\put(-.6,-.75){{\rm Figure~6. Factor lattice $\mathbf L/\Theta_p$}}
\end{picture}
\end{center}
\end{enumerate}
\end{example}

\vspace*{4mm}

In the next theorem we characterize lattices $\mathbf L$ for which the quotient lattice $\mathbf L/\Theta_p$ is modular.

\begin{theorem}
Let $\mathbf L=(L,\vee,\wedge,0)$ be a lattice with $0$ satisfying the {\rm ACC} such that $\Theta_p\in\Con\mathbf L$. Then the following are equivalent:
\begin{enumerate}[{\rm(i)}]
\item $\mathbf L/\Theta_p$ is modular,
\item for all $x,y,z\in L$, $(x\vee z)^p=z^p$ implies $\big((x\vee y)\wedge z\big)^p=\big(x\vee(y\wedge z)\big)^p$.
\end{enumerate}
\end{theorem}

\begin{proof}
The following are equivalent:
\begin{align*}
& {\rm(i)}, \\
& \text{for all }x,y,z\in L,[x]\Theta_p\le[z]\Theta_p\text{ implies} \\
& \hspace*{1cm} ([x]\Theta_p\vee[y]\Theta_p)\wedge[z]\Theta_p=[x]\Theta_p\vee([y]\Theta_p)\wedge[z]\Theta_p), \\
& \text{for all }x,y,z\in L,[x]\Theta_p\vee[z]\Theta_p=[z]\Theta_p\text{ implies }[(x\vee y)\wedge z]\Theta_p=[x\vee(y\wedge z)]\Theta_p, \\
& \text{for all }x,y,z\in L,[x\vee z]\Theta_p=[z]\Theta_p\text{ implies }\big((x\vee y)\wedge z\big)^p=\big(x\vee(y\wedge z)\big)^p, \\
& {\rm(ii)}.
\end{align*}
\end{proof}

\begin{example}
Consider the non-modular lattice $\mathbf L$ visualized in Fig.~7:
		
\vspace*{-3mm}
		
\begin{center}
\setlength{\unitlength}{7mm}
\begin{picture}(6,12)
\put(3,1){\circle*{.3}}
\put(1,3){\circle*{.3}}
\put(3,3){\circle*{.3}}
\put(5,3){\circle*{.3}}
\put(1,5){\circle*{.3}}
\put(1,7){\circle*{.3}}
\put(3,7){\circle*{.3}}
\put(5,7){\circle*{.3}}
\put(3,9){\circle*{.3}}
\put(3,11){\circle*{.3}}
\put(3,1){\line(-1,1)2}
\put(3,1){\line(0,1){10}}
\put(3,1){\line(1,1)2}
\put(1,3){\line(0,1)4}
\put(5,3){\line(0,1)4}
\put(3,9){\line(-1,-1)2}
\put(3,9){\line(1,-1)2}
\put(2.85,.3){$0$}
\put(.35,2.85){$a$}
\put(3.4,2.85){$b$}
\put(5.4,2.85){$c$}
\put(.35,4.85){$d$}
\put(.35,6.85){$e$}
\put(3.4,6.85){$f$}
\put(5.4,6.85){$g$}
\put(3.4,8.85){$h$}
\put(2.85,11.4){$1$}
\put(-4.2,-.75){{\rm Figure~7. Not pseudocomplemented non-modular lattice}}
\end{picture}
\end{center}

\vspace*{4mm}
		
The operator $^p$ is given by the following table:
\[
\begin{array}{l|c|c|c|c|c|c|c|c|c|c}
x   & 0 &    a    &    b    &    c    &    d    &    e    &    f    &    g    & h & 1 \\
\hline
x^p & 1 & \{f,g\} & \{e,g\} & \{e,f\} & \{f,g\} & \{f,g\} & \{e,g\} & \{e,f\} & 0 & 0
\end{array}
\]
Hence
\[
\Theta_p=\{0\}^2\cup\{a,d,e\}^2\cup\{b,f\}^2\cup\{c,g\}^2\cup\{h,1\}^2\in\Con\mathbf L.
\]
However, the lattice $\mathbf L/\Theta_p$ depicted in Fig.~8 is modular:
		
\vspace*{-3mm}
		
\begin{center}
\setlength{\unitlength}{7mm}
\begin{picture}(6,6)
\put(3,1){\circle*{.3}}
\put(1,3){\circle*{.3}}
\put(3,3){\circle*{.3}}
\put(5,3){\circle*{.3}}
\put(3,5){\circle*{.3}}
\put(3,1){\line(-1,1)2}
\put(3,1){\line(0,1)4}
\put(3,1){\line(1,1)2}
\put(3,5){\line(-1,-1)2}
\put(3,5){\line(1,-1)2}
\put(2.7,.3){$[0]$}
\put(.1,2.85){$[a]$}
\put(3.25,2.85){$[b]$}
\put(5.25,2.85){$[c]$}
\put(2.65,5.4){$[1]$}
\put(-.6,-.75){{\rm Figure~8. Factor lattice $\mathbf L/\Theta_p$}}
\end{picture}
\end{center}
\end{example}

Authors' addresses:

Ivan Chajda \\
Palack\'y University Olomouc \\
Faculty of Science \\
Department of Algebra and Geometry \\
17.\ listopadu 12 \\
771 46 Olomouc \\
Czech Republic \\
ivan.chajda@upol.cz

Helmut L\"anger \\
TU Wien \\
Faculty of Mathematics and Geoinformation \\
Institute of Discrete Mathematics and Geometry \\
Wiedner Hauptstra\ss e 8-10 \\
1040 Vienna \\
Austria, and \\
Palack\'y University Olomouc \\
Faculty of Science \\
Department of Algebra and Geometry \\
17.\ listopadu 12 \\
771 46 Olomouc \\
Czech Republic \\
helmut.laenger@tuwien.ac.at
\end{document}